\documentclass[11pt]{article}
\usepackage{amssymb,amsmath,amsthm,latexsym,graphicx,graphics,psfrag}

\newtheorem{theorem}{Theorem}
\newtheorem{lemma}[theorem]{Lemma}

\newtheorem{corollary}[theorem]{Corollary}
\newtheorem{conjecture}[theorem]{Conjecture}
\newtheorem{construction}[theorem]{Construction}
\newtheorem{claim}[theorem]{Claim}

\newtheorem{definition}[theorem]{Definition}
\newtheorem{problem}[theorem]{Problem}

\newtheorem{remark}[theorem]{Remark}

\numberwithin{theorem}{section}

\begin{document}
\title{\bf Coloring the squares of graphs   whose maximum average degrees are less than 4}


\author{Seog-Jin Kim \\ 
\small Department of Mathematics Education\\[-0.8ex]
\small Konkuk University\\[-0.8ex]
\small Seoul, Korea \\
\small\tt skim12@konkuk.ac.kr \\
\and
Boram Park\thanks{Corresponding author: borampark@ajou.ac.kr} \\
\small Department of Mathematics\\[-0.8ex]
\small Ajou University\\[-0.8ex]
\small Suwon, Korea\\
\small\tt borampark@ajou.ac.kr
}
\maketitle
\begin{abstract}
The square $G^2$ of  a graph $G$ is the graph defined on $V(G)$ such that two vertices $u$ and $v$ are adjacent in $G^2$ if the distance between $u$ and $v$ in $G$ is at most 2.
The {\em maximum average degree} of $G$, $mad (G)$, is the maximum among the average degrees of the subgraphs of $G$.

It is known in \cite{BLP-14-JGT} that there is no constant $C$ such that every graph $G$ with $mad(G) < 4$ has $\chi(G^2) \leq \Delta(G) + C$.
Charpentier \cite{Charpentier14} conjectured that
there exists an integer $D$ such that every graph $G$ with $\Delta(G)\ge D$ and $mad(G)<4$ has $\chi(G^2) \leq 2 \Delta(G)$. Recent result in \cite{BLP-DM} implies that   $\chi(G^2) \leq 2 \Delta(G)$ if $mad(G) < 4 -\frac{1}{c}$ with $\Delta(G) \geq 40c -16$.

In this paper, we show for $c\ge 2$, if  $mad(G) < 4 - \frac{1}{c}$ and $\Delta(G) \geq 14c-7$,  then $\chi_\ell(G^2) \leq 2 \Delta(G)$, which improves the result in \cite{BLP-DM}.
We also show that for every integer $D$, there is a graph $G$ with
$\Delta(G)\ge D$ such that $mad(G)<4$, and $\chi(G^2) \geq 2\Delta(G) +2$, which disproves Charpentier's conjecture.
In addition, we give counterexamples to Charpentier's another conjecture in \cite{Charpentier14}, stating that for every integer $k\ge 3$, there is an integer $D_k$ such that every graph $G$ with
 $mad(G)<2k$ and $\Delta(G)\ge D_k$ has  $\chi(G^2) \leq k\Delta(G) -k$.
\end{abstract}

\section{Introduction}

A proper $k$-coloring $\phi: V(G) \rightarrow \{1, 2, \ldots, k \}$ of a graph $G$ is an assignment of colors to the vertices of $G$ so that any two adjacent vertices  receive distinct colors.
The {\em chromatic number} $\chi(G)$ of a graph $G$ is the least $k$ such that there exists a proper $k$-coloring of $G$.
A {\em list assignment} on $G$ is a function
$L$ that assigns each vertex $v$ a set $L(v)$ which is
a list of available colors at $v$.
A graph $G$ is said to be {\em $k$-choosable} if for any list assignment $L$ such that
$|L(v)| \geq k$ for every vertex $v$, there exists a proper coloring $\phi$ such that $\phi(v) \in L(v)$ for every $v \in V(G)$. The  {\em list chromatic number}  $\chi_\ell(G)$ of a graph $G$ is the least $k$ such that $G$ is   $k$-choosable.

The square $G^2$ of  a graph $G$ is the graph defined on $V(G)$ such that two vertices $u$ and $v$ are adjacent in $G^2$ if the distance between $u$ and $v$ in $G$ is at most 2.  
The {\em maximum average degree} of $G$, $mad (G)$, is the maximum among the average degrees of the subgraphs of $G$.  That is, ${\displaystyle mad(G) = \mbox{max}_{H \subset G} \frac{|E(H)|}{|V(H)|}}$.

The study of $\chi(G^2)$ was initiated in \cite{Wegner}, and has been actively studied.
From the fact that $\chi(G^2) \geq \Delta(G) +1$ for every graph $G$,
 a naturally arising problem is to find graphs $G$ which  satisfy $\chi(G^2) = \Delta(G) +1$.  A lot of research has been done to find sufficient conditions in terms of by girth or $mad(G)$ to be  $\chi(G^2) = \Delta(G) +1$.
Also, given a constant $C$, determining graphs $G$ which satisfy $\chi(G^2) \leq \Delta(G) + C$ is also an interesting research topic. See \cite{BI-2009,DKNS-2008,WL2003} for more information.

Bonamy, L\'{e}v\^{e}que,   Pinlou  \cite{BLP-DM} showed that $\chi_{\ell} (G^2) \leq \Delta(G) +2$ if $mad(G) < 3$ and $\Delta(G) \geq 17$.  However, it was reported in \cite{BLP-14-JGT} that there is no constant $C$ such that every graph $G$ with $mad(G) < 4$ has $\chi(G^2) \leq \Delta(G) + C$.
On the other hand, Bonamy, L\'{e}v\^{e}que,   Pinlou \cite{BLP-14-JGT} showed the following result.
\begin{theorem}[\cite{BLP-14-JGT}] \label{Bonamy-epsilon}
There exists a function $h(\epsilon)$ such that every graph $G$ with $mad(G) < 4 - \epsilon$ satisfies $\chi_{\ell}(G^2) \leq \Delta(G) + h(\epsilon)$, where $h(\epsilon) \sim \frac{40}{\epsilon}$ as $\epsilon \rightarrow 0$.
\end{theorem}
It is known in \cite{BLP-14-JGT} that for arbitrarily large maximum degree, there exists a graph $G$ such that  $mad(G) < 4 $ and $\chi(G^2) \geq \frac{3\Delta(G)}{2}$.
On the other hand, Charpentier \cite{Charpentier14} proposed the following conjectures.
\begin{conjecture}[\cite{Charpentier14}] \label{conj-Charpentier}
There exists an integer $D$ such that every graph $G$ with $\Delta(G)\ge D$ and $mad(G)<4$ has $\chi(G^2) \leq 2 \Delta(G)$.
\end{conjecture}

\begin{conjecture}[\cite{Charpentier14}] \label{conj-k}
For each integer $k \geq 3$, there exists an integer $D_k$ such that every graph $G$ with $\Delta(G) \geq D_k$ and $mad(G) < 2k$ has $\chi(G^2) \leq k \Delta(G) -k$.
\end{conjecture}

It was mentioned in \cite{Charpentier14} that   Conjecture \ref{conj-Charpentier} and Conjecture \ref{conj-k} are best possible, if they are true.
In this paper,  we  disprove Conjecture \ref{conj-Charpentier}
by showing that for any positive integer $D$, there is a graph $G$ with
$\Delta(G)\ge D$ and $mad(G)<4$ such that $\chi(G^2) \geq 2\Delta(G) +2$.
Precisely, for arbitrarily positive integer $d\ge 2$, there exists a graph $G_d$ with $\Delta(G) = d+1$ such that $mad(G) = 4 - \frac{10}{d^2 +1}$ and the maximum clique size of $G_d^2$ is $2 \Delta(G_d) +2$.
It means that there is no constant $D_0$ such that every graph $G$ with $mad(G) < 4$  and $\Delta(G) \geq D_0$ satisfies that $\chi(G^2) \leq 2 \Delta(G)$. In addition, we give counterexamples to Conjecture \ref{conj-k} by using similar idea.

As a modification of Conjecture \ref{conj-Charpentier}, we are interested in finding the optimal value $h(c)$ such that $\chi(G^2) \leq 2 \Delta(G)$ (or $\chi_\ell(G^2) \leq 2 \Delta(G)$) for every graph $G$ with $mad(G) < 4 - \frac{1}{c}$ and $\Delta(G) \geq h(c)$.  Our main theorem of this paper is the following, which shows that $h(c) \leq 14c-7$.

\begin{theorem} \label{main-thm}
Let $c$ be an integer such that $c\ge2$.   If a graph $G$ satisfies $mad(G) < 4 - \frac{1}{c}$ and $\Delta(G) \geq 14c-7$, then $\chi_{\ell}(G^2) \leq 2 \Delta(G)$.
\end{theorem}
Note that Theorem \ref{Bonamy-epsilon} implies that if  $G$ is a graph with $mad(G) < 4 -\frac{1}{c}$ and $\Delta(G) \geq 40c -16$, then $\chi_{\ell}(G^2) \leq 2 \Delta(G)$.  Thus Theorem \ref{main-thm} gives a better bound on $h(c)$ than Theorem \ref{Bonamy-epsilon} when $14c - 7 \leq \Delta(G) \leq 40c - 17$.

Next, 
we will show that $h(c) \geq 2c+2$.  Thus the current bound on $h(c)$ is
$2c+2 \leq h(c) \leq 14c -7$.  Hence it would be interesting  to solve the following problem.

\begin{problem} \label{problem-h(c)}\rm Given a positive integer $c\ge 1$,  there is a function $h(c)$ such that
 $\chi(G^2) \leq 2 \Delta(G)$ (or $\chi_{\ell}(G^2)\leq 2\Delta(G)$) whenever a graph $G$ satisfies $mad(G) < 4 -\frac{1}{c}$ and $\Delta(G) \geq h(c)$.
What is the optimal value of $h(c)$?  Or, reduce the gap in $2c+2 \leq h(c) \leq 14c -7$.
\end{problem}

\begin{remark}\rm Yancey \cite{Yancey15} showed that for $t \geq 3$, if $G$ is a graph  with $mad(G) < 4 - \frac{4}{t+1} - \epsilon$ for some $\frac{4}{t(t+1)} > \epsilon > 0$, then $\chi_{\ell}(G^2) \leq
\max\{ \Delta(G) + t, \ 16t^2 \epsilon^{-2}\}$.
We can convert $mad(G) < 4 - \frac{4}{t+1} - \epsilon$ into $mad(G) < 4 - \frac{1}{c}$ form by setting
$\frac{4}{t+1} + \epsilon = \frac{1}{c}$.  Then from $0 < \epsilon  < \frac{4}{t(t+1)} < 1$, we have that
 $    \epsilon< \frac{1-c\epsilon}{c} \times  \frac{1 - c\epsilon }{4c +c\epsilon-1}$, and consequently $0< \epsilon  < \frac{1}{c(4c+1)} < 1$.  Thus $0 < 1 - c \epsilon < 1$, and consequently, we have
$t = \frac{4c}{1-c\epsilon}-1 > 4c -1$.  Hence, when $16t^2 \epsilon^{-2} \le   2\Delta(G)$, we have
\[
\Delta(G) \geq 8t^2 \epsilon^{-2} > 8 (4c -1)^2 c^2 (4c+1)^2
\]
since $\epsilon  < \frac{1}{c(4c+1)}$.  Thus Yancey's result implies that $\chi (G^2) \leq 2 \Delta(G)$ only when $mad(G) < 4 - \frac{1}{c}$ and $\Delta(G) \geq t_0 c^6$ for some constant $t_0$.
But, note that in our result, the lower bound on $\Delta(G)$ is linear as $\Delta(G) \geq 14c-7.$
\end{remark}

This paper is organized as follows. In Section \ref{sec2}, we will give a construction which is a counterexample to Conjecture \ref{conj-Charpentier}, and  in Section \ref{section-main}, we will prove Theorem \ref{main-thm} using discharging method.
In Section \ref{section-remark}, we modify the construction in Section \ref{sec2} slightly, and show that for any positive integer $c$, there exists a graph $G$ such that $mad(G) < 4 - \frac{1}{c}$, $\Delta(G) = 2c+1$, and $\chi(G^2) = 2 \Delta(G) + 1$, which implies that $h(c) \geq 2c+2$. And next, in Appendix, we will  give counterexamples to Conjecture \ref{conj-k}.


\section{Construction}\label{sec2}

We will show that for any positive integer $n\ge 2$, there is a graph $G$ with
$\Delta(G)= n+1$ such that $mad(G)<4$, and $\chi(G^2)>2\Delta(G)$.
Let $[n] = \{1, \ldots, n\}$.

\bigskip

\begin{construction} \label{construction-one}\rm
Let $n \geq 2$ be a positive integer. Let $S=\{u_1,u_2,\ldots,u_n\}$, $T=\{v_1,\ldots,v_n\}$, and $X=\{x_{ij} \mid (i,j) \in [n] \times [n] \}$.
We define a graph $G_n$ by
 \begin{eqnarray*}
  V(G_n)&=&\{u,v\}\cup S \cup T \cup X \\
  E(G_n)&=&\{uv\}\cup \{uu_i\mid u_i\in S\} \cup \{vv_i\mid v_i\in T\} \\&&\cup
  \left( \bigcup_{i=1}^{n} \bigcup_{j=1}^{n}  \{u_ix_{ij},v_jx_{ij} \}\right) \cup
  \left( \bigcup_{i=2}^{n} \{ x_{11}x_{ii}\}\right) \cup
  \left( \bigcup_{i=2}^{n} \{ x_{12}x_{i(i+1)}\}\right),
\end{eqnarray*}
where 
 $x_{n (n+1)} = x_{n1}$. See Figure~\ref{fig1} for an illustration.
 \begin{figure}
 \centering
\psfrag{a}{\footnotesize$u_1$}
\psfrag{b}{\footnotesize$u_2$}
\psfrag{c}{\footnotesize$u_3$}
\psfrag{d}{\footnotesize$u_{n-1}$}
\psfrag{e}{\footnotesize$u_n$}
\psfrag{f}{\footnotesize$v_1$}
\psfrag{g}{\footnotesize$v_2$}
\psfrag{h}{\footnotesize$v_3$}
\psfrag{i}{\footnotesize$v_{n-1}$}
\psfrag{j}{\footnotesize$v_n$}
\psfrag{u}{\footnotesize$u$}
\psfrag{v}{\footnotesize$v$}
\psfrag{k}{\footnotesize$x_{11}$}
\psfrag{l}{\footnotesize$x_{22}$}
\psfrag{m}{\footnotesize$x_{33}$}
\psfrag{n}{\footnotesize$x_{n-1,n-1}$}
\psfrag{o}{\footnotesize$x_{nn}$}
\psfrag{p}{\footnotesize$x_{12}$}
\psfrag{q}{\footnotesize$x_{23}$}
\psfrag{r}{\footnotesize$x_{n-1,n}$}
\psfrag{s}{\footnotesize$x_{n1}$}
\includegraphics[width=10cm]{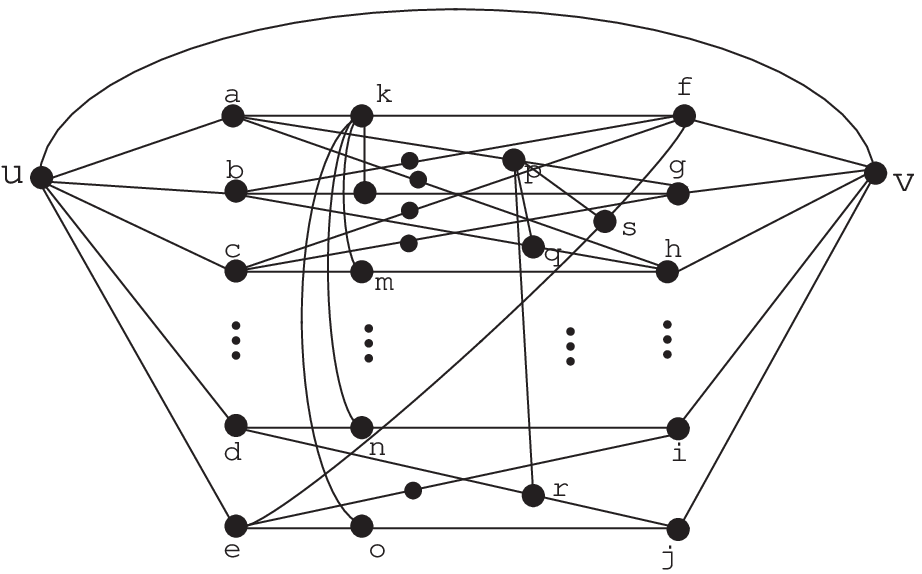}\\
  \caption{Construction of $G_n$}\label{fig1}
\end{figure}
\end{construction}
We have the following simple observations.
\begin{itemize}
\item For $y\in \{u,v\}\cup S\cup T$, $d(y)=n+1$.
\item For $x_{ij}\in X$,
\[d(x_{ij})=\left\{\begin{array}{ll} n+1 & \text{if } (i, j) = (1, 1), \mbox{ or } (1, 2) \\
  3 & \text{if }j=i+r \text{ for }r\in\{0,1\}\text{ and }i\ge 2,  \\
  2 & \text{otherwise, }
\end{array}\right.\]
where   $x_{n (n+1)} = x_{n1}$.
\end{itemize}
Therefore $\Delta(G_n)=n+1$ and $\{u,v,x_{11}, x_{12}\}\cup S\cup T$ is a clique in $G_n^2$ with $2n+4$ vertices.
Thus $\chi(G_n^2) \geq 2\Delta(G_n) +2$.

From now on,
we denote $G_n$ by $G$ for simplicity.
Next, we will show that $mad(G)<4$.  Denote  the number of edges of the subgraph of $G$ induced by $A$ by
$||A||$, that is,
$|E(G[A])| = ||A||$.
Define a potential function $\rho_G:2^{V(G)}\rightarrow \mathbb{Z}$ by for $A\subset V(G)$,
\[
\rho_G (A) = 2|A| - ||A||.
\]
Note that $\rho_G(A) \geq 1$ for every subset $A \subset V(G)$ is equivalent to $mad(G) < 4 $.

\bigskip

We will show that $\rho_G(A)\ge 1$ for all $A\subset V(G)$.  A vertex of degree $k$ is called a $k$-vertex, and a vertex of degree at least $k$ (at most $k$) is called a $k^+$-vertex ($k^{-}$-vertex).

\begin{claim} \label{claim-construction}
For all $A\subset V(G)$, $\rho_G(A)\ge 1$.
\end{claim}
\begin{proof}
Suppose that there is $A\subset V(G)$ such that  $\rho_G(A)\le 0$.
Let $A$ be a smallest subset of $V(G)$ among all subsets of $V(G)$ with minimum value $\rho_G(A)$.
That is, $A$ is a minimal counterexample to Claim \ref{claim-construction}.

If $G[A]$ contains a $2^{-}$-vertex $v$ then $\rho_G(A\setminus\{v\})\le \rho_G(A)$, which is a contradiction to the minimality of $\rho_G(A)$ or the minimality of $|A|$. Thus $G[A]$ does not have a $2^{-}$-vertex.

Let $X_3$ be the set of $3^+$-vertices in $X$.
Then every vertex in $X  \setminus X_3$ does not belong to $A$, since each vertex in $X  \setminus X_3$ is a 2-vertex.

If $a \notin A$ and $a$ has at least three neighbors in $A$, then $\rho_G(A\cup\{a\})<\rho_G(A)$, a contradiction  to the minimality of $\rho_G(A)$.  Thus
every vertex not in $A$ has at most two neighbors in $A$.

Next, we will show that $\{u,v\}\subset A$.
If $|A\cap S|\le 1$, then any vertex in $A\cap T$ is a $2^-$ vertex of $G[A]$, a contradiction.
Thus $|A\cap S|\ge 2$.
Suppose that $|A\cap S|= 2$. If  $v \not\in A$, then $A\cap T$ has a 2-vertex of $G[A]$, which is forbidden.
Thus $v\in A$, and then $u$ is adjacent to three vertices of $A$, and so $u\in A$. Therefore $\{u,v\} \subset A$.
Similarly, if $|A\cap T|= 2$, then $\{v,v\} \subset A$.
On the other hand, if $|A\cap S|\ge 3$ and $|A\cap T|\ge 3$, then $\{v,v\} \in A$, since
every vertex not in $A$ has at most two neighbors in $A$.
Therefore, we can conclude that $\{u,v\}\subset A$.

Let $X_3'$ be the set of 3-vertices of $G$ in $X\cap A$.  That is,
$X_3' = (X_3 \cap A) \setminus \{x_{11}, x_{12}\}$.
As we noted that every vertex in $X  \setminus X_3$ does not belong to $A$,
in fact, $X_3' = (X \cap A) \setminus \{x_{11}, x_{12}\}$.
Note that every vertex in $X_3'$ is also a 3-vertex of $G[A]$.
Since any two vertices in $X_3'$ are not adjacent in $G$, we have
\[ \rho_{G}(A\setminus X_3') = 2|A\setminus X_3'|-||A\setminus X_3'|| = 2|A|-2|X_3'|-||A||+ 3||X_3'|| = \rho_G(A) +|X_3'|.\]
Since $\rho_G(A)\le 0$,
\begin{eqnarray}\label{eq-claim2}
&&\rho_{G}(A\setminus X_3')  \le |X_3'|.
\end{eqnarray}
Note that each vertex in $(A\setminus  X_3')\cap X = A \cap \{x_{11}, x_{12}\}$ has degree at most $2$ in $G[A\setminus X_3']$, and therefore, we have
$\rho_{G}(A\setminus X_3') \ge \rho_{G}(A\setminus X)$.

Let $\alpha=|A\cap (S\cup T)|$ for simplicity.
Note that $\alpha \ge |X_3'|$, since for vertex $x$ in $X_3'$, $x$ is a $3$-vertex in both $G$ and $G[A]$, and so $N_G(x)\subset A$.
Now note that $G[A\setminus X]$ has $\alpha+2$ vertices and has $\alpha+1$  edges.
Thus
\[ \rho_{G}(A\setminus X_3') \ge \rho_{G}(A\setminus X) \ge  2\alpha + 4-(\alpha+1) \ge \alpha +3 \ge |X_3'| +3,\]
a contradiction to \eqref{eq-claim2}.
Therefore $\rho_G(A) \geq 1$ for every subset $A \subset V(G)$.
This completes the proof of Claim \ref{claim-construction}.
\end{proof}

\begin{remark}\rm In Appendix,
we  will also show that Conjecture \ref{conj-k} is not true. That is, for any integers $k$ and $n$ such that $k\ge 2$ and $n\ge k^2-k$, there exists a graph $G$ such that $mad(G)<2k$, $\Delta(G)\ge n$, and $\chi(G^2)\ge k\Delta(G)+k$. The construction for $k \geq 3$ is  similar to
Construction \ref{construction-one}.
\end{remark}

\section{Proof of Theorem \ref{main-thm}} \label{section-main}

We use double induction on the number of $3^{+}$-vertices first, and then on the number of edges.

\begin{definition}
Let $n_3(G)$ be the number of $3^{+}$-vertices of $G$. We order graphs as follows.
Give two graphs $G$ and $G'$, say that $G'$ is {\em smaller than} $G$ if  (1) $n_3(G') < n_3 (G)$, or
(2) $n_3(G') = n_3 (G)$ and $|E(G')| < |E(G)|$.
\end{definition}

Throughout this section, we let $G$ be a minimal counterexample to Theorem \ref{main-thm}.

\begin{lemma} \label{nbr-lemma}
If a vertex $u$ has a neighbor of degree 2, then \[\displaystyle \sum_{x \in N(u)} d(x) \geq 2 \Delta(G).\]
\end{lemma}

\begin{proof} Suppose that $\displaystyle \sum_{x \in N(u)} d(x) < 2 \Delta(G)$.
Let $v$ be a neighbor of $u$ whose degree is 2.
Let $H=G-uv$.
The number of $3^{+}$-vertices of $H$ is not greater than that of $G$, and the number of edges of $H$ is less than that of $G$.  Thus, $\chi_{\ell}(H^2)\le 2 \Delta(H)$.
Note that $2\Delta(H)\le 2\Delta(G)$.
Now uncolor $u$ and $v$. Then the number of forbidden colors at $v$ is at most $2\Delta(G)-1$ and so we can assign a color of $v$.
And the number of forbidden colors at $u$ is at most $\sum_{x \in N(u)} d(x) <2\Delta(G)$, and so we can give a color to $u$. Thus $G^2$ is $2\Delta(G)$-choosable.  This is a contradiction.
\end{proof}

\begin{corollary}\label{nbr-coro}
Let $u$ be a vertex having a neighbor of degree 2. Then
\begin{itemize}
\item[(i)] if $d(u)\le \frac{2\Delta(G)}{3}$, then $u$ has at least one neighbor of degree at least $4$;
\item[(ii)] if $d(u) \le \frac{\Delta(G)}{3}$, then  $u$ has at least two neighbors of degree at least $4$;
\item[(iii)] a $2$-vertex is not adjacent to a $2$-vertex.
\end{itemize}
\end{corollary}

\begin{lemma} \label{3-lemma}
Every $3$-vertex has a neighbor of degree   at least $4c$.
\end{lemma}
\begin{proof}   
For any graph $G'$, we define a potential function $\rho_{G'}:2^{V(G')}\rightarrow \mathbb{Z}$ by
\[
\rho_{G'} (A) = (4c -1)|A| - 2c||A||.
\]
Note that $\rho_{G'}(A) \geq 1$ for every subset $A \subset V(G')$ is equivalent to $mad(G') < 4 - \frac{1}{c}$.

Let $u$ be a $3$-vertex of $G$, and let $N(u) = \{x_1, x_2, x_3 \}$.
Suppose that $max \{d(x_1), d(x_2), d(x_3)\} \leq 4c -1$.
Let $H$ be the graph obtained by deleting the vertex $u$ and adding three vertices $y_1, y_2, y_3$ such that $N(y_1) = \{x_1, x_2\}$,  $N(y_2) = \{x_2, x_3\}$, and $N(y_3) = \{x_1, x_3\}$.

Now we will show that $\rho_H (S) \geq 1$ for every $S \subset V(H)$.
Suppose that  there exists $S\subset V(H)$ such that $\rho_H (S)\le 0$. We take a smallest such $S$.
If $S \cap \{y_1, y_2, y_3\} = \emptyset$, then $\rho_H (S) = \rho_G (S) \geq 1$.
Thus $S \cap \{y_1, y_2, y_3\} \neq \emptyset$.
If a vertex $y$ in $S\cap \{y_1,y_2,y_3\}$ is a $1^-$ vertex, then $\rho_H (S) > \rho_H(S\setminus\{y\})$, a contradiction to the minimality of $|S|$.
Thus, any vertex in $S\cap \{y_1,y_2,y_3\}$ is a  $2$-vertex.

Let $S'= S \setminus \{y_1, y_2, y_3\} $ and $|S\cap \{y_1,y_2,y_3\}|=\alpha$.
If $\alpha=1$ then $\rho_H (S) = \rho_G (S' \cup \{u\}) \geq 1$.
If  $\alpha\ge 2$, then  $\{x_1,x_2,x_3\} \subset S$ and so
\begin{eqnarray*}
 \rho_H(S) &=& \rho_H(S') + (4c-1)\alpha - 4c\alpha \\&=& \rho_G(S') -\alpha\\
&=&\rho_G(S'\cup\{u\}) -(4c-1) +6c -\alpha\\&=& \rho_G(S'\cup\{u\}) + 2c-\alpha+1\ge 1,
\end{eqnarray*}
where the last inequality is from the fact that $\rho_G(S'\cup\{u\})\ge1$ and $\alpha\le 3$.
Hence $\rho_H (S) \geq 1$ for every $S \subset V(H)$.

Note that each $x_i$ has degree in $H$ at least 3 by Lemma~\ref{nbr-lemma}.
Hence the number $3^+$-vertices of $H$ is smaller than the number of $3^+$-vertices of $G$.
Thus 
by the minimality of $G$, we have
$\chi_{\ell} (H^2) \leq 2 \Delta(H)$.
Since the degrees of $x_1$, $x_2$, $x_3$ in $G$ are at most $4c-1$ and $\Delta(G)\ge 14c-7$,
$\Delta(H)=\Delta(G)$. Thus $\chi_\ell(H^2) \leq 2 \Delta(G)$.
Now, since the number of 2-distance neighbors of $u$ is at most $12c < 2 \Delta(G)$, the number of forbidden colors at  $u$ is less than $2\Delta(G)$.
Thus $G^2$ is $2 \Delta(G)$-choosable.  This is a contradiction.  This completes the proof of Lemma \ref{3-lemma}.
\end{proof}

\bigskip
\noindent{\bf Discharging Rules} \\

\noindent {\bf R1}: If $d(u) \geq 8c -2$, then $u$ sends $1 - \frac{1}{2c}$ to  each of its neighbors. \\

\noindent {\bf R2}: If $4c \leq d(u) < 8c -2$, $u$ sends $1 - \frac{1}{2c}$ to each of neighbors of degree 2, and sends $1 - \frac{1}{c}$ to each of neighbors of degree 3.  \\

\noindent {\bf R3}: If $4 \leq d(u) < 4c $ and $u$ has exactly one neighbor of degree at least $4$, then $u$
does not send any charge to its neighbors.  \\

\noindent {\bf R4}: If $4 \leq d(u) < 4c $ and $u$ has at least two neighbors of degree at least 4,  then $u$ sends
$1 - \frac{1}{2c}$ to each of its neighbors of degree 2.  \\

\noindent {\bf R5}: If a $3$-vertex $u$ has two neighbors of degree at least $8c-2$ and one neighbor of degree 2, then $u$ sends $1 - \frac{1}{2c}$ to its neighbor  whose degree is 2.

\bigskip

\bigskip

Let $d^*(u)$ be the new charge after discharging.  We will show  that $d^*(u) \geq 4 - \frac{1}{c}$ for all $u$.  Note that $\Delta(G) \geq 14c -7$.

\bigskip
(1) When $d(u) \geq 8c -2$,
\[
d^*(u) \geq d(u) - d(u)\left(1 - \frac{1}{2c}\right) = \frac{d(u)}{2c} \geq 4 - \frac{1}{c}.
\]

\bigskip
(2) If $4c\leq d(u) \leq 8c -3$ and $u$ has no neighbor of degree 2,
then \[   d^*(u) \geq d(u) - d(u)\left(1 - \frac{1}{c}\right) =   \frac{d(u)}{c}   \geq 4 - \frac{1}{c} .\]

\bigskip
(3) Suppose that  $4c\leq d(u) \leq 8c -3$ and $u$ is adjacent to a 2-vertex.
Note that by (i) of Corollary~\ref{nbr-coro}, $u$ is adjacent to at least one $4^+$-vertex.
\begin{itemize}
\item If $6c-1 \leq d(u) \leq 8c -3$, then
\[
d^*(u) \geq d(u) - (d(u) -1)\left(1 - \frac{1}{2c}\right) = 1 + \frac{d(u)}{2c} - \frac{1}{2c} \geq 4 - \frac{1}{c},
\]
since $u$ is adjacent to at least one $4^+$-vertex.
\item If $4c \leq d(u) \leq 6c-2$ and $u$ has exactly one neighbor $z$ of degree at least 4,
then
\[\displaystyle \sum_{x \in N(u)} d(x) \le  d(z)+3 \cdot (d(u)-2)+2, \]
and by Lemma~\ref{nbr-lemma},
\[ 2 \Delta(G)\le \sum_{x \in N(u)} d(x)    \le d(z)+3\cdot (d(u)-2)+2  .\]
Thus
\[ d(z) \ge  2 \Delta(G) -3\cdot(d(u)-2)-2 \ge2\cdot(14c-7) -3\cdot(6c-4)-2\ge 8c-2,\]
which implies that $u$ receives charge $1-\frac{1}{2c}$ from $z$. Thus
\[
d^*(u) \geq d(u) - (d(u) -1)\left(1 - \frac{1}{2c}\right) + 1-\frac{1}{2c} = 2 + \frac{d(u)}{2c} - \frac{1}{c} \geq 4 - \frac{1}{c}.
\]
\item If $4c \leq d(u) \leq 6c-2$ and $u$ is adjacent to at least two $4^+$-vertices, then
\[
d^*(u) \ge d(u) - (d(u) -2)\left(1 - \frac{1}{2c}\right) = 2 + \frac{d(u)}{2c} - \frac{1}{c} \geq 4 - \frac{1}{c}.
\]
\end{itemize}
Thus
$d^*(u) \geq  4 - \frac{1}{c}$.

\bigskip

(4) Suppose that $2c+1 \leq d(u) < 4c $.
If $u$ has no neighbor of degree 2, then $u$ does not send any charge to others.
Next, consider the case when $u$ has  a neighbor of degree 2.
By (ii) of Corollary~\ref{nbr-coro}, $u$ is adjacent to at least two $4^{+}$-vertices.
\begin{itemize}
\item Suppose that $u$ has exactly two neighbors of degree at least 4, say $z_1$ and $z_2$.
Then by Lemma~\ref{nbr-lemma},
\[ d(z_1) + d(z_2) + 3\cdot (d(u)-2) \geq 2 \Delta(G).
\]  Note that
\[
d(z_1) + d(z_2) \geq  2 \Delta(G) - 3d(u) +6  \geq 2\cdot (14c-7)- 3\cdot(4c-1) + 6 =16c-5=2\cdot(8c-2)-1.
\]
  Thus at least one of $d(z_1)$ and $d(z_2)$ is at least $8c-2$, implies that
  $u$ receives charge at least $1 - \frac{1}{2c}$ from $z_1$ and $z_2$.  Thus
\[
d^*(u) \geq d(u) - (d(u)-2)\left(1 - \frac{1}{2c}\right) +  1 - \frac{1}{2c}
= \frac{d(u)}{2c} + 3 - \frac{3}{2c} \geq 4 - \frac{1}{c}.
\]
\item If $u$ is adjacent to at least three $4^{+}$-vertices, then
\[
d^*(u) \geq d(u) - (d(u)-3)\left(1 - \frac{1}{2c}\right)
= \frac{d(u)}{2c} + 3 - \frac{3}{2c} \geq 4 - \frac{1}{c}.
\]
\end{itemize}

\bigskip

(5) Suppose that $4 \leq d(u) < 2c + 1$.
If  $u$ has no neighbor of degree 2,
then $u$ does not send any charge to others.
Consider the case when $u$ has  a neighbor of degree 2.
If $u$ has at most $(d(u)-4)$ neighbors of degree 2, then
\[
d^*(u) \geq d(u) - (d(u)-4)\left(1 - \frac{1}{2c}\right)
= \frac{d(u)}{2c} + 4 - \frac{2}{c} \geq 4 - \frac{1}{c}.
\]
Suppose that $u$ has at least $(d(u)-3)$ neighbors of degree 2.
Let $z_1$, $z_2$ and $z_3$ be the other neighbors.
By Lemma~\ref{nbr-lemma}, since $u$ has a neighbor of degree 2,
\[ d(z_1) + d(z_2) + d(z_3)+2\cdot(d(u)-3) \geq 2 \Delta(G).\]
Note that
\[
d(z_1) + d(z_2) + d(z_3) \geq  2 \Delta(G) - 2d(u) +6  \geq 2\cdot (14c-7) - 2\cdot2c + 6 \geq 3 \cdot (8c-2)-2.
\]
Thus we can conclude that at least one of $d(z_1)$, $d(z_2)$, $d(z_3)$ is at least $8c-2$, and so $u$ receives charge at least $1 - \frac{1}{2c}$ from $z_1, z_2, z_3$.  Thus
\[
d^*(u) \geq d(u) - (d(u)-3)\left(1 - \frac{1}{2c}\right) + 1 - \frac{1}{2c}
= \frac{d(u)}{2c} + 4 - \frac{4}{2c} \geq 4.
\]

\bigskip
(6)  When $d(u) = 3$, by Lemma~\ref{3-lemma},
$u$ has at least one neighbor of degree at least $4c$. Thus $u$ receives charge at least $1 - \frac{1}{c}$ from its neighbors. Even though we consider  {\bf R5},  we have  $d^*(u) \geq 4 -\frac{1}{c}$.

\bigskip
(7) Suppose that $d(u) = 2$.
We will show that $u$ receives  $1 - \frac{1}{2c}$ from both neighbors.
Let $x$ be a neighbor of $u$.
Suppose that $d(x)\le 3$. Then $d(x)=3$ by (iii) of Corollary~\ref{nbr-coro}. By Lemma~\ref{nbr-lemma},  each neighbor of $x$ other than $u$ has degree at least $\Delta(G)-2$, which implies that $u$ receives  charge $1 - \frac{1}{2c}$ from $x$ by {\bf R5}.
Suppose that $d(x)\ge 4$.
If $4 \leq d(x) < 4c $ and $x$ has exactly one neighbor of degree at least 4, the it violates (ii) of Corollary~\ref{nbr-coro}.
Thus the case of {\bf R3} does not happen to $x$. That is, $x$ must send $1 - \frac{1}{2c}$ to $u$.
Then $d^*(u) \geq 4 -\frac{1}{c}$.


\section{Remark on a condition for $\Delta(G)$ to be $\chi(G^2) \le 2 \Delta(G)$} \label{section-remark}

Given a positive integer $c\ge 2$, let $h(c)$ be the smallest value such that $\chi(G^2) \leq 2 \Delta(G)$ whenever $G$ is a graph with $mad(G) < 4 - \frac{1}{c}$ and $\Delta(G) \geq h(c)$. Theorem \ref{main-thm} shows that $h(c) \leq 14c-7$.
In the following, we will see that $h(c)\ge 2c+2$ by showing that for any integer $c\ge 2$, there is a graph $G$ such that $mad(G)<4-\frac{1}{c}$, $\Delta(G) = 2c +1$, and $\chi(G^2) \geq 2 \Delta(G) +1$.
Hence $2c+2 \le  h(c) \leq 14c -7$. Thus,
it would be interesting to find  the optimal value of $h(c)$ or to reduce the gap in $2c+2 \le  h(c) \leq 14c -7$.

Now, given a positive integer $c\ge 2$,  we give a graph $G$ such that $mad(G)<4-\frac{1}{c}$, $\Delta(G) = 2c +1$, and $\chi(G^2) \geq 2 \Delta(G) +1$.
Let consider $G_n$ in  Section~\ref{sec2} when $n = 2c$, and then let \[
G = G_{n} -  \left(\{x_{12}x_{i(i+1)} : 2 \leq i \leq n \} \cup \{ x_{12}x_{n1} \} \right).\]
Then $|V(G)| = 4c^2 + 4c +2$ and $|E(G)| = 8c^2 + 6c$.
Therefore, $\Delta(G)=n+1$ and $\{u,v,x_{11}\}\cup S\cup T$ is a clique in $G^2$ with $2n+3$ vertices.
Thus $\chi(G^2) \geq 2\Delta(G) +1$.

\begin{claim}
$mad (G) < 4 - \frac{1}{c}$.
\end{claim}
\begin{proof}[Proof of Claim]
Define a potential function $\rho^*_G (A) = (4c -1)|A| - 2c||A||$ for $A \subset V(G)$.
Note that $mad (G) < 4 - \frac{1}{c}$ is equivalent to  $\rho^*_G(A)\ge 1$ for all $A\subset V(G)$.

Now, we will show that $\rho^*_G(A)\ge 1$ for all $A\subset V(G)$.
Suppose that there is $A\subset V(G)$ such that  $\rho^*_G(A)\le 0$, and take such $A$ with minimum value $\rho^*(A)$. If $G[A]$ contains a $1^{-}$-vertex $v$ then $\rho^*_G(A\setminus\{v\})<\rho^*_G(A)$, which is a contradiction to the minimality of $\rho^*(A)$. Thus $G[A]$ does not have a $1^{-}$-vertex.
If $a\not\in A$ and $a$ has at least two neighbors in $A$, then $\rho^*(A\cup\{a\})<\rho^*(A)$, a contradiction  to the minimality of $\rho^*(A)$.
Therefore, if $a \notin A$, then $a$ has at most one neighbor in $A$.
Thus for $i,j$ such that $i\neq j$,
\begin{eqnarray}\label{eqij}
x_{ij}\in A  \Longleftrightarrow \{u_i,v_j\}\subset A.
\end{eqnarray}
Without loss of generality, we may assume that $|S\cap A| \leq |T\cap A|$.
From \eqref{eqij}, it is easy to check that if $|A\cap S|\le 1$ or $|A \cap T| \leq 1$, then $\rho^*(A)\ge 1$. 
Thus we can assume that $|T\cap A| \geq  |S \cap A|\ge 2$,  and so $\{u,v\} \in A$.

For simplicity, let $s=|S\cap A|$ and $t=|T\cap A|$.  Then $s \leq t$.
By~\eqref{eqij}, $|A\cap X|=|S\cap A|\cdot |T \cap A|=st$.  Thus we have that
$|A|\ge st+s+t+2$.
On the other hand,
\begin{eqnarray*}||A||\leq 2st+s+t + 1 + (s-1) =2 \cdot (st+s+t+2)-(t+4).
\end{eqnarray*}
Thus,
\begin{eqnarray*}
\rho^*(A)&=&(4c-1)|A|-2c||A|| \\
&\ge &(4c-1)(st+s+t+2)-2c\cdot \left( 2 \cdot (st+s+t+2)-(t+4) \right)\\
&=& (4c-1)(st+s+t+2)-4c \cdot(st+s+t+2)+2c \cdot (t+4)\\
&=& -(st+s+t+2)+2c(t+4)\\
&=& -st-s-t-2+2tc+8c\\
&=& (2ct-st)+(2c-s)+(2c-t)+(4c-2) \ge 1,
\end{eqnarray*}
where the last equality is from the fact that $2c \ge \max\{s,t\}$.
This is a contradiction to the assumption that
$\rho^*(A)\le 0$.  Thus $\rho^*(A) \ge 1$ for every subset $A \subset V(G)$.
\end{proof}


\section*{Appendix: Counterexamples to Conjecture~\ref{conj-k}}

Fix an integer $k\ge 3$ and let $n$ be a  prime with $n\ge k^2-k$.
We will define a graph $G=G_{k,n}$ such that $\Delta(G)=k+n-1$, $mad(G^2)<2k$, and $G^2$ contains a clique of size $k\Delta(G)+k$.
The idea is exactly same as Construction \ref{construction-one} in Section 2.

For $\ell \in [k-2]$,  we define a Latin square $L_{\ell}$ of order $n$ by
\begin{eqnarray*} \label{eq:Latin}
 L_{\ell}(i, j)= j+\ell (i-1) \pmod{n}, \quad \mbox{ for } (i, j) \in [n] \times [n].
\end{eqnarray*}
That is, the $(i, j)$-entry of the Latin square of $L_{\ell}$ is $L_{\ell}(i, j)$.
(See page 252 in \cite{Van-Lint} for detail.)

\medskip

\begin{construction} \label{construction-general}\rm
For $i\in [k]$,
let $U_i=\{u_{i,1},u_{i,2},\ldots,u_{i,n}\}$. Let  $U=\{u_1,u_2,\ldots, u_k\}$ and $X=\{x_{i,j} \mid (i,j) \in [n] \times [n] \}$.
Define
\begin{eqnarray*}
  V(G)&=& U \cup \left(\bigcup_{i=1}^{n} U_i\right) \cup X \\
  E(G)&=& \{u_iu_j \mid 1\le i< j\le k \} \cup \left(\bigcup_{i=1}^{k} \{u_ix\mid x \in U_i \}\right) \\&&\cup
  \left( \bigcup_{i=1}^{n} \bigcup_{j=1}^{n}  \left\{ x_{i,j}y\mid y\in \{u_{1,i},u_{2,j},u_{3,L_{1}(i,j)}, u_{4,L_{2}(i,j)},\ldots,u_{k,L_{k-2}(i,j)}\}  \right\}\right) \\
  &&\cup  \left( \bigcup_{r=0}^{k^2-k-1} \bigcup_{i=2}^{n} \{ x_{1,1+r}  x_{i,i+r}  \}  \right)
\end{eqnarray*}
where the subscripts of $x_{i,j}$ are computed by modulo $n$.
\end{construction}

Then we have the following observations.
\begin{itemize}
\item For $u_i\in U$,  $d(u_i)=n+k-1$ and for $u_{i,j}\in U_i$, $d(u_{i,j})=n+1.$
\item For $x_{i,j}\in X$,
\[d(x_{i,j})=\left\{\begin{array}{ll} 2k-1 & \text{if } i=1 \\
  k+1 & \text{if }j=i+r  \mbox{ for some  } 0\le r \le  k^2-k -1 \text{ and }i\ge 2 \\
  k & \text{otherwise }
\end{array}\right.\]
where  the subscript of $x_{i,j}$ are computed by modulo $n$.
\end{itemize}
Therefore $\Delta(G)=n+k-1$.

\begin{claim}
$\chi(G^2) \geq k\Delta(G) +k$.
\end{claim}

\begin{proof} We will show that
$\{x_{1,1}, x_{1,2}, x_{1,3},\ldots,x_{1,k^2-k}\}\cup U\cup U_1\cup \cdots \cup U_{k}$ is a clique in $G^2$.
From the orthogonality of Latin squares, we know that   $u_{i,j}$ and $u_{i',j'}$ are adjacent in $G^2$ if $i\neq i'$.
For each $i\in [k]$, since   $u_{i,j}$ and $u_{i,j'}$ share a neighbor $u_i$, they are adjacent in $G^2$.
In addition, $u_i$ and $u_{i',j}$ are adjacent in $G^2$, since they share a neighbor $u_{i'}$.
Therefore, $U\cup U_1\cup \cdots \cup U_{k}$ is a clique in $G^2$.

Note that the vertices in $\{x_{1,1}, x_{1,2}, \ldots,x_{1,k^2-k}\}$ share a neighbor $u_{1,1}$, and so
they form a clique in $G^2$. Furthermore, each vertex in $U$  is adjacent to each vertex in $X$ in $G^2$ since they share a neighbor in $U_1\cup \cdots \cup U_{k}$.
Thus, it remains to show that for each integer $r$ such that $0\le r \le  k^2-k -1$,
$x_{1,1+r}$ is adjacent to each vertex in $U_1\cup \cdots \cup U_{k}$.

Let $r$ be an integer with  $0\le r \le  k^2-k -1$.
Since for $i\in [n]$,
\[ N_G(x_{i,i+r}) \supset \{u_{1,i},u_{2,i+r},u_{3,L_{1}(i,i+r)}, u_{4,L_{2}(i,i+r)},\ldots,u_{k,L_{k-2}(i,i+r)}\}.\]
Thus  $N_G(x_{1,1+r})\cup N_G(x_{2,2+r}) \cup \cdots\cup N_G(x_{n,n+r})$ contains
\begin{eqnarray*}
&&\{u_{1,1},u_{2,1+r},u_{3,L_{1}(1,1+r)}, u_{4,L_{2}(1,1+r)},\ldots,u_{k,L_{k-2}(1,1+r)}\}\\
&\cup&  \{u_{1,2},u_{2,2+r},u_{3,L_{1}(2,2+r)}, u_{4,L_{2}(2,2+r)},\ldots,u_{k,L_{k-2}(2,2+r)}\}\\
&&\vdots\\
&\cup&
\{u_{1,n},u_{2,n+r},u_{3,L_{1}(n,n+r)}, u_{4,L_{2}(n,n+r)},\ldots,u_{k,L_{k-2}(n,n+r)}\}.
\end{eqnarray*}
Since for each $\ell\in [k-2]$,
\[ \{ L_{\ell}(1,1+r) ,L_{\ell}(2,2+r),\ldots ,L_{\ell}(n,n+r) \}=[n],\]
we can conclude that
\[ N_G(x_{1,1+r})\cup N_G(x_{2,2+r}) \cup \cdots\cup N_G(x_{n,n+r}) \supset U_1\cup \cdots \cup U_{k}. \]
Since $x_{1,1+r}$ is adjacent to every vertex in $\{x_{2,2+r},x_{3,3+r}, \ldots, x_{n,n+r}\}$,
$x_{1,1+r}$ is adjacent to each vertex in  $U_1\cup \cdots \cup U_{k}$ in $G^2$.

Consequently,
$\{x_{1,1}, x_{1,2}, x_{1,3},\ldots,x_{1,k^2-k}\}\cup U\cup U_1\cup \cdots \cup U_{k}$ is a clique in $G^2$ with $kn+k+(k^2-k)=k\Delta(G)+k$ vertices.
Thus $\chi(G^2) \geq k\Delta(G) +k$.\end{proof}

Next, we will show that $mad(G)<2k$.
Define a potential function $\rho_G:2^{V(G)}\rightarrow \mathbb{Z}$ by for $A\subset V(G)$,
\[
\rho_G (A) = k|A| - ||A||.
\]
Note that $\rho_G(A) \geq 1$ for every $A \subset V(G)$ is equivalent to $mad(G) <2k$.

\bigskip

Now, we will show that $\rho_G(A)\ge 1$ for all $A\subset V(G)$.

\begin{claim} \label{mad-2k}
For all $A\subset V(G)$, $\rho_G(A)\ge 1$.\end{claim}

\begin{proof}
Suppose that there is $A\subset V(G)$ such that  $\rho_G(A)\le 0$.
Let $A$ be a smallest subset of $V(G)$ among all subsets of $V(G)$ with minimum value $\rho_G(A)$.
That is, $A$ is a minimal counterexample to Claim \ref{mad-2k}.

If there is a $k^{-}$-vertex $v$ of $G[A]$, then $\rho_G(A\setminus\{v\})\le \rho_G(A)$, which is a contradiction to the minimality of $\rho_G(A)$ or the minimality of $|A|$. Thus there is no $k^{-}$-vertex in $G[A]$.
Thus if a vertex $x$ in $X \cap A$ has degree $ k+1$ in $G$, then $N_G(x)\subset A$.

Therefore if $x \in X$ has degree $k$ in $G$, then $x\not\in A$, and
if a vertex $x$ in $X \cap A$ has degree $ k+1$ in $G$, then $N_G(x)\subset A$.
Let $X'$ be the set of $(k+1)$-vertices of $G$ in $X\cap A$.
Thus every vertex in $X'$ is also a $(k+1)$-vertex in $G[A]$.
Since any two vertices in $X'$ are not adjacent in $G$, we have
\[ \rho_{G}(A\setminus X') = k|A\setminus X'|-||A\setminus X'|| = k|A|-k|X'|-||A||+ (k+1)||X'|| = \rho_G(A) +|X'|.\]
Since $\rho_G(A)\le 0$,
\begin{eqnarray}\label{eq-claim4}
\rho_{G}(A\setminus X') \le |X'|.
\end{eqnarray}

On the other hand,  all the vertices in $(A\setminus X')\cap X$ have degree at most $k$ in $G[A\setminus X']$. Then $\rho_G(A\setminus X') \ge \rho_G(A\setminus X)$.
Then all those vertices in $A \setminus (X\cup U)$ are pendent vertices in $G[A\setminus X]$.
Let $\alpha=|A\setminus (X\cup U)|=|A\cap (U_1\cup \cdots U_k)|$ for simplicity.
Note that $\alpha \ge |X'|$, since for vertex $x$ in $X'$, $N_G(v)\subset A$.
Let $u=|A\cap U|$.
Therefore, $G[A\setminus X]$ has $u+\alpha$ vertices and has $\frac{u^2-u}{2}+\alpha$  edges.
Thus
\begin{eqnarray*}
   \rho_G(A\setminus X') \ge \rho_G(A\setminus X) &\ge&  ku+k\alpha-\frac{u^2-u}{2}-\alpha \\
   &\ge& u^2+k\alpha-\frac{u^2-u}{2}-\alpha = \frac{u^2+u}{2}+(k-1)\alpha \\
&\ge& 1+(k-1)|X'| \ge 1+|X'| ,
\end{eqnarray*}
a contradiction to \eqref{eq-claim4}.
Thus  $\rho_G(A) \geq 1$ for every subset $A \subset V(H)$.
Hence, $mad(G)<2k$.
\end{proof}

\end{document}